\documentclass[12pt,a4paper,final]{amsart}
 
\usepackage[notcite,notref,color]{showkeys} 

\usepackage{amsmath,amssymb,amsthm,bbm} 
\usepackage{enumitem} 
\usepackage[overload]{textcase}
 
\usepackage{epsfig,psfrag,color} 

\usepackage[colorlinks,bookmarks,linkcolor=black,citecolor=black]{hyperref} 
 
\usepackage[english,algoruled,lined,noresetcount,norelsize]{algorithm2e}

\setlist{itemsep=1pt,parsep=0pt,topsep=2pt,partopsep=0pt}  
\setenumerate{leftmargin=*} 

\def\itm#1{\rm ({#1})} 
\def\itmit#1{\itm{\it #1\,}} 
\def\rom{\itmit{\roman{*}}} 
\def\abc{\itmit{\alph{*}}} 
 
\def\endofFact{\hfill\scalebox{.6}{$\Box$}}

\SetTitleSty{textbf}{}
\SetArgSty{textrm}

\SetKw{Failure}{failure}
\SetKw{Return}{return}
\SetKw{NOT}{not}
\SetKwFor{RepeatForever}{repeat}{forever}{end}

\let\subset\subseteq  
\let\eps\varepsilon 
\let\rho\varrho 
\def\dcup{\dot\cup}  

\def\cC{\mathcal{C}}
\def\cD{\mathcal{D}}
\def\cE{\mathcal{E}}
\def\cF{\mathcal{F}}
\def\cG{\mathcal{G}}
\def\cH{\mathcal{H}}
\def\cP{\mathcal{P}}

\newtheorem{theorem}{Theorem}
\newtheorem{lemma}[theorem] {Lemma}

\newtheorem{question}[theorem] {Question}  
\theoremstyle{remark}  
\newtheorem{remark}[theorem] {Remark}  
\newtheorem{claim}[theorem] {Claim}  

\newcommand{\oldqed}{}
\newenvironment{claimproof}[1][Proof]{
  \renewcommand{\oldqed}{\qedsymbol}
  \renewcommand{\qedsymbol}{\endofFact}
  \begin{proof}[#1]
}{
  \end{proof}
  \renewcommand{\qedsymbol}{\oldqed}
} 

\newcommand{\By}[2]{\overset{\mbox{\tiny{#1}}}{#2}} 
\newcommand{\ByRef}[2]{   \By{\eqref{#1}}{#2} }

\newcommand{\eqByRef}[1]{ \ByRef{#1}{=} }

\newcommand{\leByRef}[1]{ \ByRef{#1}{\le} } 
\newcommand{\geByRef}[1]{ \ByRef{#1}{\ge} }

\newcommand{\NATS}{\mathbb{N}} 

\newcommand{\Prob}{\mathbb{P}}
\newcommand{\Exp}{\mathbb{E}}

\newcommand{\mult}{\text{mult}}

\newcommand{\Gr}[1][r]{\cG^{(#1)}}
\newcommand{\Crl}[1][\ell]{\cC^{(r)}_{#1}}
\newcommand{\Drl}[1][\ell]{\cD^{(r)}_{#1}}
\newcommand{\mo}{m^{(1)}}
\newcommand{\Bad}{B}

\newcommand{\EMAIL}[1]{  \textit{E-mail}: \texttt{#1} }

\newcommand{\tpl}[1]{\mathbf{#1}}

\title{Tight Hamilton cycles in random hypergraphs}

  \author[P. Allen]{Peter Allen*}
  \author[J. B\"ottcher]{Julia B\"ottcher*}

  \thanks{
    *
    Department of Mathematics, London School of Economics, Houghton Street,
London WC2A 2AE, U.K.
   \EMAIL{p.d.allen|j.boettcher@lse.ac.uk}
 }
  \author[Y. Kohayakawa]{Yoshiharu Kohayakawa\dag}
\thanks{
    \dag
    Instituto de Matem\'atica e Estat\'{\i}stica, Universidade de
    S\~ao Paulo, Rua do Mat\~ao 1010, 05508--090~S\~ao Paulo, Brazil.
   \EMAIL{yoshi@ime.usp.br}
 }
  \author[Y. Person]{Yury Person\ddag}
  \thanks{
    \ddag  
Institut f\"ur Mathematik, Freie Universit\"at Berlin, Arnimallee 3-5, D-14195 Berlin, Germany.
   \EMAIL{person@math.fu-berlin.de}
  }

  \thanks{
    PA was partially supported by FAPESP (Proc.~2010/09555-7).
    JB was partially supported by FAPESP (Proc.~2009/17831-7).
    YK was partially supported by CNPq (308509/2007-2, 477203/2012-4),
CAPES/DAAD (415/ppp-probral/po/D08/11629, 333/09) and NUMEC (Project
MaCLinC/USP).
    YP was partially supported by GIF grant no.~I-889-182.6/2005.
    The cooperation of the authors was
    supported by a joint CAPES-DAAD project (415/ppp-probral/po/D08/11629,
    Proj.~no.~333/09).
    The authors are grateful to NUMEC/USP, N\'ucleo de Modelagem Estoc\'astica e
    Complexidade of the University of S\~ao Paulo, for supporting this
    research.
  }

\date{\today}

\begin{document}
\begin{abstract}
  We give an algorithmic proof for the existence of tight Hamilton cycles
  in a random $r$-uniform hypergraph with edge probability $p=n^{-1+\eps}$
  for every $\eps>0$. This partly answers a question of Dudek and Frieze
  [Random Structures Algorithms], who used a second moment method to show
  that tight Hamilton cycles exist even for $p=\omega(n)/n$ ($r\ge 3$) where $\omega(n)\to\infty$ arbitrary slowly,   
  and for $p=(e+o(1))/n$ ($r\ge 4$).

  The method we develop for proving our result applies to related problems
  as well.
\end{abstract}
\maketitle

\section{Introduction}\label{sec:intro}

The question of when the random graph~$G(n,p)$ becomes hamiltonian is well
understood. P\'osa~\cite{Posa} and Korshunov~\cite{Kor76, Kor77} 
proved that the hamiltonicity threshold is
$\log n/n$, 
Koml\'os and Szemer\'edi~\cite{KomSzem} determined an exact
formula for the probability of the existence of a Hamilton cycle, and
Bollob\'as~\cite{Boll} established an even more powerful hitting time result.
The first polynomial time randomised algorithms for finding Hamilton cycles
in $G(n,p)$ were developed by Angluin and Valiant~\cite{AngVal} and
Shamir~\cite{Shamir}. Finally, Bollob\'as, Fenner and Frieze~\cite{BFF}
gave a deterministic polynomial time algorithm whose success probability
matches the probabilities established by Koml\'os and Szemer\'edi.

For random hypergraphs much less is known.  The random $r$-uniform
hypergraph $\cG^{(r)}(n,p)$ on vertex set~$[n]$ is generated by including
each hyperedge from $\binom{[n]}{r}$ independently with probability
$p=p(n)$.  First, Frieze~\cite{Frieze} considered loose Hamilton cycles in
random $3$-uniform hy\-per\-graphs. The \emph{loose $r$-uniform cycle} on
vertex set $[n]$ has edges $\{i+1,\dots,i+r\}$ for exactly all $i=k(r-1)$
with $k\in\NATS$ and $(r-1)\mid n$, where we calculate modulo~$n$. Frieze showed that the
threshold for a loose Hamilton cycle in $\cG^{(3)}(n,p)$ is $\Theta(\log
n/n^2)$. Dudek and Frieze~\cite{DudFriLoose} extended this to $r$-uniform
hypergraphs with $r\ge 4$, where the threshold is $\tilde{\Theta}(\log
n/n^{r-1})$. Both results require that $n$ is divisible by $2(r-1)$ (which was 
recently removed by Dudek, Frieze, Loh and Speiss~\cite{DFLS12}) and
rely on the deep Johansson-Kahn-Vu theorem~\cite{JKV}, which makes their
proofs non-constructive.

Tight Hamilton cycles, on the other hand, were first considered in
connection with packings.  The \emph{tight $r$-uniform cycle} on vertex set
$[n]$ has edges $\{i+1,\dots,i+r\}$ for all $i$ 
calculated modulo~$n$.  Frieze, Krivelevich and Loh~\cite{FriKriLoh} proved
that if $p\gg (\log^{21}n/n)^{1/16}$ and~$4$ divides~$n$ then most edges of
$G^{(3)}(n,p)$ can be covered by edge disjoint tight Hamilton
cycles. Further packing results were obtained by Frieze and Krivelevich~\cite{FriKri11} and by
Bal and Frieze~\cite{BalFri12}, but the probability range is far from best possible.  
Subsequently, Dudek and Frieze~\cite{DudFriTight} used a second
moment argument to show that the threshold for a tight Hamilton cycle in
$\cG^{(r)}(n,p)$ is sharp and equals $e/n$ for each $r\ge 4$ and for $r=3$ 
they showed that  
$\cG^{(3)}(n,p)$ contains a tight Hamilton cycle when $p=\omega(n)/n$ for any 
 $\omega(n)$ that goes to infinity. Since their method is
non-constructive they asked for an algorithm to find a tight Hamilton cycle
in a random hypergraph. In this paper we present a randomised algorithm for
this problem if $p$ is slightly bigger than in their result.

\begin{theorem}\label{thm:main} 
  For each integer $r\ge 3$ and $0<\eps<1/(4r)$ there is a randomised
  polynomial time algorithm which for any $n^{-1+\eps}<p\le 1$ a.a.s. finds
  a tight Hamilton cycle in the random $r$-uniform hypergraph
  $\cG^{(r)}(n,p)$.
\end{theorem}

The probability referred to in Theorem~\ref{thm:main} is with respect to the random 
bits used by the algorithm as well as by $\cG^{(r)}(n,p)$. 
The running time of the algorithm in the above theorem is
polynomial in~$n$, where the degree of the polynomial depends on~$\eps$.

\smallskip

\paragraph{\bf Organisation.}
We first provide some notation and a brief sketch of our proof, formulate
the main lemmas and prove Theorem~\ref{thm:main} in
Section~\ref{sec:proof}. In Sections~\ref{sec:connect} and~\ref{sec:reserve}
we prove the main lemmas, and in Section~\ref{sec:conclude} we end with
some remarks and open problems.


\section{Lemmas and proof of Theorem~\ref{thm:main}}
\label{sec:proof}

\subsection{Notation}
An $s$-\emph{tuple} $(u_1,\dots,u_s)$ of vertices is an ordered set of
vertices. We often denote tuples by bold symbols, and occasionally also
omit the brackets and write $\tpl{u}=u_1,\dots,u_s$. Additionally, we may
also use a tuple as a set and write for example, if~$S$ is a set,
$S\cup\tpl{u}:=S\cup\{u_i\colon i\in[s]\}$.  The \emph{reverse} of the
$s$-tuple $\tpl{u}$ is the $s$-tuple $(u_s,\dots,u_1)$. 


In an $r$-uniform hypergraph~$\cG$ the tuple $P=(u_1,\dots,u_\ell)$ forms a
\emph{tight path} if $\{u_{i+1},\dots,u_{i+r}\}$ is an edge for every $0\le
i\le \ell-r$.  For any $s\in[\ell]$ we say that~$P$ \emph{starts} with the
$s$-tuple $(u_1,\dots,u_s)=:\tpl{v}$ and \emph{ends} with the $s$-tuple
$(u_{\ell-(s-1)},\dots,u_\ell)=:\tpl{w}$. We also call~$\tpl{v}$ the
\emph{start $s$-tuple} of $P$, $\tpl{w}$ the \emph{end $s$-tuple} of
$P$, and~$P$ a $\tpl{v}-\tpl{w}$ path. The \emph{interior} of~$P$ is formed
by all its vertices but its start and end $(r-1)$-tuples. Note that the interior
of $P$ is not empty if and only if $\ell>2(r-1)$.

For a hypergraph~$\cH$ we define the \emph{$1$-density} of~$\cH$ to be
$d^{(1)}(\cH):=e(\cH)/\big( v(\cH) -1\big)$ if $v(\cH)>1$, and $d^{(1)}(\cH):=0$
if $v(\cH)=1$. We set
\begin{equation*}
  \mo(\cH):=\max\{d^{(1)}(\cH') \colon \cH'\subset \cH\} \,.
\end{equation*}
We denote the $r$-uniform tight cycle on~$\ell$ vertices by~$\Crl$.
Observe that $\mo(\Crl)=\ell/(\ell-1)$.

\subsection{Outline of the proof}

A simple greedy strategy shows that for $p=n^{\eps-1}$ it is easy to find a
tight path (and similarly a tight cycle) in~$\cG^{(r)}(n,p)$ which covers
all but at most $n^{1-\frac12\eps}$ of its vertices. Incorporating these few
remaining vertices is where the difficulty lies.

To overcome this difficulty we apply the following strategy, which we call the
\emph{reservoir method}.  We first construct a tight path $P$ of a linear length in $n$ which
contains a vertex set $W^*$, called the \emph{reservoir}, such that for any
$W\subseteq W^*$ there is a tight path on $V(P)\setminus W$ whose end 
$(r-1)$-tuples are 
the same as that of $P$.  In a second step we use the mentioned greedy
strategy to extend~$P$ to an almost spanning tight path $P'$, with a leftover set
$L$.  The advantage we have gained now is that we are permitted to reuse
the vertices in $W^*$: we will show that, by using a subset~$W$ of vertices
from~$W^*$ to incorporate the vertices from~$L$, we can extend the almost
spanning tight path to a spanning tight cycle~$C$. More precisely, we shall
delete~$W$ from~$P'$ (observe that, by construction of $P$, the hypergraph
 induced on $V(P)\setminus W$ contains a tight path with the same ends)
  and use precisely all vertices of~$W$ to connect the
vertices of~$L$ to construct $C$.

We remark that our method has similarities, in spirit, with the absorbing
method for proving extremal results for large structures in dense
hypergraphs (see, e.g., R\"odl, Ruci\'nski and
Szemer\'edi~\cite{RRS}). 
The techniques to deal with multi-round exposure in our algorithm is
similar to those used by Frieze in~\cite{Fri88}.
Moreover, a method very similar to ours was used
independently by K\"uhn and Osthus~\cite{KOPosa} to find bounds on the
threshold for the appearance of the square of a Hamilton cycle in a random
graph.

\subsection{Lemmas}

We shall rely on the following lemmas. We state these lemmas together with
an outline of how they are used, and then give the details of the proof of
Theorem~\ref{thm:main}.

Our first lemma asserts that there are hypergraphs~$\cH^*$ with density
arbitrarily close to~$1$ which have a spanning tight path and a
vertex~$w^*$ such that deleting~$w^*$ from~$\cH^*$ leaves a spanning tight
path with the same start and end $(r-1)$-tuples.

\begin{lemma}[Reservoir lemma]\label{lem:reserve} 
  For all $r\ge 2$ and $0<\eps<1/(6r)$, there exist an $r$-uniform hypergraph
  $\cH^*=\cH^*(r,\eps)$ on less than $16/\eps^2$ vertices, 
  a vertex~$w^*$, and two disjoint
  $(r-1)$-tuples $\tpl{u}=(u_1,\ldots,u_{r-1})$ and
  $\tpl{v}=(v_1,\ldots,v_{r-1})$  such that 
  \begin{enumerate}[label=\rom]
  \item\label{lem:reserve:1} $m^{(1)}(\cH^*)\le 1+\eps$,
  \item\label{lem:reserve:2} $\cH^*$ has a tight Hamilton $\tpl{u}-\tpl{v}$
    path, and
  \item\label{lem:reserve:3} $\cH^*-w^*$ has a tight Hamilton
    $\tpl{u}-\tpl{v}$ path.
 \end{enumerate}
\end{lemma}

We provide a proof of Lemma~\ref{lem:reserve} in Section~\ref{lem:reserve}.
We also call the graph~$\cH^*$ asserted by this lemma the \emph{reservoir
  graph} and the vertex~$w^*$ the \emph{reservoir vertex}, since they will
provide us as follows with the reservoir mentioned in the outline. If we
can find many disjoint copies of $\cH^*$ in $\cG^{(r)}(n,p)$, and if we can
connect these copies of~$\cH^*$ to form a tight path, then the set~$W^*$ of
reservoir vertices~$w^*$ from these $\cH^*$-copies forms such a reservoir.

In order to find many disjoint $\cH^*$-copies, we use the following
standard theorem.

\begin{theorem}[see, e.g.,  {\cite[Theorem 4.9]{JaLuRu:Book}}] \label{thm:disjcopy}
  For every $r$-uniform hy\-per\-graph~$\cH$ there are constants $\nu>0$ and
  $C\in\NATS$ such that if $p\ge C
  n^{-1/m^{(1)}(\cH)}$, then $\Gr(n,p)$ a.a.s.\ contains $\nu n$ vertex disjoint
  copies of $\cH$.
\end{theorem}

For connecting the $\cH^*$-copies into a long tight
path~$P$ we use the next lemma.

\begin{lemma}[Connection lemma]\label{lem:connect} 
  Given $r\ge 3$, $0<\eps<1/(4r)$ and $\delta>0$, there exists $\eta>0$ 
  such that there is a (deterministic) polynomial time
  algorithm~$\mathcal{A}$ which on inputs $\cG=\cG^{(r)}(n,p)$ with $p= n^{-1+\eps}$ a.a.s.\ does
  the following.

  Let $1\le k\le \eta n$, let $X$ be any subset of $[n]$ of size at least
  $\delta n$.  Let $\tpl{u}^{(1)},\ldots,\tpl{u}^{(k)}$,
  $\tpl{v}^{(1)},\ldots,\tpl{v}^{(k)}$ be any $2k$ pairwise disjoint
  $(r-1)$-tuples in~$[n]$. Then~$\mathcal{A}$ finds in~$\cG$ a collection
  of vertex disjoint tight paths $P_i$, $1\le i\le k$, of length at most
  $\ell:=(r-1)/\eps+2$, such that $P_i$ is a $\tpl{u}^{(i)}-\tpl{v}^{(i)}$ path all of
  whose interior vertices are in $X$.
\end{lemma}

 We prove this lemma in Section~\ref{sec:connect}.
 In fact, we will also make use of this lemma after extending $P$ to a maximal tight
 path $P'$ in order to extend $P'$ (reusing vertices of the reservoir $W^*$) to
 cover the leftover vertices $L$. It is for this
 reason that we  require the lemma to work with a set $X$ which
 can be quite small.

\subsection{Proof of the main theorem}

Our goal is to describe an algorithm which a.a.s.\ constructs
a tight Hamilton cycle in the $r$-uniform random hypergraph $\Gr(n,q)$ in
five steps. For convenience we replace $\Gr(n,p)$ in Theorem~\ref{thm:main} 
by $\Gr(n,q)$. 
 We make use of a $5$-round exposure of the
random hypergraph, that is, each of the five algorithm steps will
individually a.a.s.\ succeed on an $r$-uniform random hypergraph with edge
probability somewhat smaller than~$q$.
Observe, however, that for the algorithm the input graph is given at once,
and not as the union of five graphs. Therefore, in a preprocessing step the
algorithm will first split the (random) input hypergraph into five
(independent random) hypergraphs. The only probabilistic component of our algorithm is 
in the preprocessing step. 

Our five algorithm steps will then be as follows.  Firstly, we apply
Theorem~\ref{thm:disjcopy} in order to find $cn$ vertex disjoint copies of
the reservoir graph $\cH^*$ from Lemma~\ref{lem:reserve}.  Secondly, we use
the connection lemma, Lemma~\ref{lem:connect}, to connect the $\cH^*$
copies to a tight path $P$ of length $c'n$ which contains a set $W^*$ of
linearly many reservoir vertices. Thirdly, we greedily extend $P$ until we
get a tight path $P'$ on $n-n^{1-(\eps'/2)}$ vertices. In the fourth and
fifth step we use~$W^*$ and Lemma~\ref{lem:connect} to connect the
remaining vertices to the path constructed so far and to close the path
into a cycle.

For technical reasons it will be convenient to assume that the edge
probability in each of the last four steps is exactly $q'=n^{-1+\eps'}$ for
some $\eps'$. We therefore split our random input hypergraph into five
independent random hypergraphs, of which the first has edge
probability~$q''\ge q'$ and the remaining four have edge probability $q'$.

\begin{proof}[Proof of Theorem~\ref{thm:main}]

  {\sl Constants:}
  Given $r\ge 3$ and $0<\eps<1/(4r)$, set $\eps':=\eps/2$.  Suppose in the
  following that~$n$ is sufficiently large and define
  $q'=n^{-1+\eps'}$. Now let $q>n^{-1+\eps}$ be given and observe that
  $q\ge 5 q' \ge 1-(1-q')^5$. Finally, let $q''\in(0,1]$ be such
  that 
  \begin{equation}\label{eq:main:p}
    1-q=(1-q'')(1-q')^4
  \end{equation}
  and note that since $q\ge 1-(1-q')^5$, we have $q''\ge q'$.

  Let $\eta_1>0$ be the constant given by Lemma~\ref{lem:connect}
  with input $r$, $\eps'$ and $\delta=1/2$. Let $\cH^*=\cH^*(r,\eps'/2)$ be the
  $r$-uniform reservoir hypergraph given by Lemma~\ref{lem:reserve} and
  $n^*:=v(\cH^*)$. Let $\nu>0$ be the
  constant given by Theorem~\ref{thm:disjcopy} with input $\cH^*$. We set
  \begin{equation}\label{eq:defc}
    c:=\min\Big(\frac{1}{2n^*},\frac{\nu}{n^*},\eta_1\Big)\,. %
  \end{equation}
  Finally, let $\eta_2>0$ and $\ell_2$ be the constants given by
  Lemma~\ref{lem:connect} with input $r$, $\eps'$ and $\delta=c/2$.

  \smallskip

  {\sl Preprocessing:}
  We shall use a randomised procedure to split the input
  graph~$\cG$ which is distributed according to $\Gr(n,q)$
  into five hypergraphs $\cG_1,\ldots,\cG_5$, such that $\cG_1$ is
  distributed according to $\Gr(n,q'')$ and
  $\cG_2,\dots,\cG_5$ are distributed according to $\Gr(n,q')$, where the
  choice of parameters is possible by~\eqref{eq:main:p}.
  Moreover these five random hypergraphs are mutually independent.
 
  Our randomised procedure takes
  a copy $\cG$ of $\Gr(n,q)$ and colours its edges as follows.
  It colours each edge~$e$ of~$\cG$ independently with a non-empty subset~$c$ of $[5]$
  such that
  \begin{equation*}
    \Pr(e\text{ receives colour }c)=\begin{cases}
      q'^{|c|}(1-q')^{4-|c|}(1-q'')/q & \text{ if } 1\notin c\\
      q'^{|c|-1}(1-q')^{5-|c|}q''/q& \text{ if } 1\in c \,.
    \end{cases}
  \end{equation*} 
  Then we let $\cG_i$ be the hypergraph with those edges whose colour
  contains~$i$ for each $i\in[5]$.

  For justifying that this randomised procedure has the desired effect, let
  us consider the following second random experiment. We take five
  independent random hypergraphs, $\cG_1=\Gr(n,q'')$ and four copies
  $\cG_2,\ldots,\cG_5$ of $\Gr(n,q')$, and form an $r$-uniform hypergraph
  on $n$ vertices, whose edges are the union of $\cG_1,\ldots,\cG_5$, each
  receiving a colour which is a subset of $[5]$ identifying the subset of
  $\cG_1,\ldots,\cG_5$ containing that edge.  Observe that we simply obtain
  $\Gr(n,q)$, when we ignore the colours in this union.

  It is straightforward to check that the two experiments yield identical
  probability measures on the space of $n$-vertex coloured
  hypergraphs. It follows that any algorithm which with some probability
  finds a tight Hamilton cycle when presented with the five hypergraphs
  $\cG_i$ of the first experiment succeeds with the same probability when
  presented with five hypergraphs obtained from the second experiment.

  \smallskip

  {\sl Step 1:}
  The first main step of our algorithm finds $cn$ vertex disjoint copies of the reservoir
  graph $\cH^*$  in $\cG_1$. To this end we would like to apply
  Theorem~\ref{thm:disjcopy}, hence we need to check its preconditions. We
  require that $q''\ge Cn^{-1/m^{(1)}(\cH^*)}$ for some large $C$. By
  Lemma~\ref{lem:reserve} we have $m^{(1)}(\cH^*)\le 1+\frac12\eps'$, and
  $1/(1+\frac12\eps')>1-\eps'$. It follows that for all sufficiently large $n$
we
  have $q'=n^{-1+\eps'}\ge Cn^{-1/m^{(1)}(\cH^*)}$, and so the same holds for
  $q''$ since $q''\ge q'$.

  By Theorem~\ref{thm:disjcopy} and~\eqref{eq:defc}, a.a.s.\ $\cG_1$
  contains at least $\nu n\ge n^*\cdot cn$ vertex disjoint copies of
  $\cH^*$. Hence we can algorithmically \emph{find} a subset of at least
  $cn$ of them as follows.  We search the vertex subsets of size~$n^*$
  of~$G_1$. Whenever we find a subset that induces $\cH^*$ and does not
  share vertices with a previously chosen $\cH^*$-copy, then we choose it.
  Clearly, we can do this until we chose $cn$ vertex disjoint copies
  $\cH_1,\ldots,\cH_{cn}$ of~$\cH^*$.  This requires running time
  $O\big(n^{n^*}\big)$, where $n^*\le 16\eps^{-2}$ does not depend
  on $n$.

  \smallskip

  {\sl Step 2:} The second step consists of using $\cG_2$ and
  Lemma~\ref{lem:connect} with input $r,\eps'$ and $\delta=1/2$ to connect
  the $cn$ vertex disjoint reservoir graphs into one tight path. Let $W^*$
  consist of the $cn$ reservoir vertices, one in each of
  $\cH_1,\ldots,\cH_{cn}$. By~\eqref{eq:defc}, $\cH_1,\ldots,\cH_{cn}$
  cover at most $n/2$ vertices. By Lemma~\ref{lem:connect} applied with
  $X=[n]\setminus\big(\bigcup_{i\in[cn]} V(\cH_i)\big)$ there is a
  polynomial time algorithm which a.a.s.\ for each $1\le i\le cn-1$ finds a
  tight path in $\cG_2$ connecting the end $(r-1)$-tuple of $\cH_i$ with
  the start $(r-1)$-tuple of $\cH_{i+1}$, where these tight paths are
  disjoint and have their interior in~$X$.  This yields a tight path $P$ in
  $\cG_1\cup\cG_2$ containing all of the $\cH_i$ with the following
  property.  For any $W\subset W^*$, if we remove $W$ from $P$, then we
  obtain (using the additional edges of the $\cH_i$) a tight path $P(W)$
  whose start and end $(r-1)$-tuples are the same as those of $P$ (see
  Lemma~\ref{lem:reserve}\ref{lem:reserve:1}).

  \smallskip

  {\sl Step 3:} In the third step we use $\cG_3$ to greedily extend~$P$
  to a tight path $P'$ covering all but at most $n^{1-\frac12\eps'}$ vertices.
  Let $P_0=P$ and do the following for each $i\ge 0$. Let $\tpl{e}_i$ be
  the end $(r-1)$-tuple of~$P_i$ if there is an edge $\tpl{e}_iv_i$ in
  $\cG_3$ for some $v_i\in[n]\setminus V(P_i)$ then append~$v_i$ to~$P_i$ to
  obtain the tight path $P_{i+1}$. If no such edge exists, then halt.

  Observe that in step~$i$ of this procedure, it suffices to reveal the
  edges $\tpl{e}_iw$ with $w\in[n]\setminus P_i$. Hence, by the method of
  deferred decision, the probability that $v_i$ does not exist is at most
  $(1-q')^{n-|P_i|}$. So, as long as $|P_i|\le n-n^{1-\frac12\eps'}$ this
  probability is at most $\exp(-q'n^{1-\frac12\eps'})\le \exp(-n^{\frac12\eps'})$.
  We take the union bound over all (at most $n$) $i$ to infer that this procedure
  a.a.s.\  indeed terminates with a tight path $P'$ 
  with $|P'|\ge n-n^{1-\frac12\eps'}$ which contains~$P$.

  \smallskip

  {\sl Step 4:}
  Now let $L'$ be the set of vertices not covered by $P'$. Let~$L$ be obtained from
  $L'$ by adding at most $r-2$ vertices of $W^*$, such that $|L|$ is divisible
  by $r-1$. Let $Y_1,\ldots,Y_t$ be a partition of $L$ into $|L|/(r-1)$ tuples
  of size $r-1$. Let $Y_0$ be the reverse of the
  start $(r-1)$-tuple of $P'$, and $Y_{t+1}$ be the reverse of its end
  $(r-1)$-tuple. 
  
  In the fourth step, we use $\cG_4$ and Lemma~\ref{lem:connect} with input
  $r$, $\eps'$ and $\delta=\frac12c$ to find for each $0\le i\le \frac12t$
  a tight path between $Y_{2i}$ and $Y_{2i+1}$ of length at most $\ell_2$
  using only vertices in $W^*\setminus L$, such that these paths are
  pairwise disjoint. This is possible since $|W^*\setminus L|\ge \frac12cn$
  and since $t\le|L|\le n^{1-\frac12\eps'}+r-2$ implies $\frac{t}2+1\le
  n^{1-\frac13\eps'} \le \eta_2 n$ for~$n$ sufficiently large. Let~$W^{**}$ be
  the set of at least $cn-(\frac{t}2+1)\ell_2\ge cn-n^{1-\frac13\eps'}\ell_2\ge
  \frac23cn$ vertices in~$W^*$ not used in this step.

  \smallskip

  {\sl Step 5:} Similarly, in the fifth step, we use $\cG_5$ and
  Lemma~\ref{lem:connect}, with input $r$, $\eps$ and $\delta=c/2$, to find
  for each $0\le i\le \frac12(t-1)$ a tight path between $Y_{2i+1}$ and
  $Y_{2i+2}$ of length at most $\ell_2$ using only vertices in
  $W^{**}\setminus L$, such that these paths are pairwise disjoint.  Again,
  $|W^{**}\setminus L|\ge \frac12cn$ and $\frac{t}2+1\le \eta_2 n$ for $n$
  sufficiently large.  Thus Lemma~\ref{lem:connect} guarantees that
  this step a.a.s.\ succeeds also and the tight paths can be found in polynomial
  time.
  
  \smallskip

  But now we are done: Let~$W$ be the vertices of~$W^*$ used in steps~4
  and~5. By definition of~$W^*$ we can delete the vertices of~$W$
  from~$P'$ and obtain a tight path $P'(W)$ through the remaining vertices
  of~$P'$ (using additional edges of the reservoir graphs) and with the
  same start and end $(r-1)$-tuples. Then $P'(W)$
  together with the connections constructed in steps~$4$ and~$5$ (which
  incorporated all vertices of~$L$) form a
  Hamilton cycle in~$\cG$.
\end{proof}

\begin{remark}\label{rem:complexity}
  We note that the only non-deterministic part of the algorithm presented
  in the above proof concerns the partition of the edges of the input graph
  into five random subsets at the beginning.

  The algorithm in the connection lemma (Lemma~\ref{lem:connect}) is
  polynomial time, where the power of the polynomial is independent
  of~$\eps$.  The same is (obviously) true for the greedy procedure of
  step~3.  Finding many vertex disjoint reservoir graphs in step~1 however,
  we can only do in time~$n^{16\eps^{-2}}$.
\end{remark}

\section{Proof of the connection lemma}
\label{sec:connect}

{\bf Preliminaries.}
For a binomially
distributed random variable~$X$ and a constant~$\gamma$ with $0<\gamma\le 3/2$
we will use the following Chernoff bound, which can be found, e.g., in~\cite
[Corollary~2.3]{JaLuRu:Book}:
\begin{equation}\label{eq:chernoff}
  \Prob\big(\, |X-\Exp X|\ge\gamma \Exp X\big)\le 2\exp(-\gamma^2\Exp X/3)
  \,.
\end{equation}
In addition we apply the following consequence of Janson's inequality (see
for example~\cite{JaLuRu:Book}, Theorem 2.18): Let~$\cE$ be a finite set
and~$\cP$ be a family of non-empty subsets of~$\cE$. Now consider the
random experiment where each $e\in\cE$ is chosen independently with
probability~$p$ and define for each~$P\in\cP$ the indicator variable $I_P$
that each element of~$P$ gets chosen. Set~$X=\sum_{P\in\cP} I_P$ and
$\Delta=\frac12\sum_{P\neq P',P\cap
  P'\neq\emptyset}\Exp(I_PI_{P'})$. Then
\begin{equation}\label{eq:Janson}
  \Prob(X=0)\le\exp(\Delta - \Exp X)
  \,.
\end{equation}

For $e\in\binom{n}{r}$ we say that we \emph{expose the $r$-set~$e$} in
$\cG^{(r)}(n,p)$, if we perform (only) the random experiment of including~$e$
in~$\cG^{(r)}$ with probability~$p$ (recall that $p:=n^{-1+\eps}$). If this
random experiment includes~$e$ then
we say that \emph{$e$ appears}.  Clearly, we can iteratively generate (a
subgraph of) $\cG^{(r)}(n,p)$ by exposing $r$-sets, as long as we do not expose
any $r$-set twice. For a tuple~$\tpl{u}$ of at most $r-1$ vertices in~$[n]$ we say
that we \emph{expose the $r$-sets at~$\tpl{u}$}, if we expose all $r$-sets
$e\in\binom{n}{r}$ with $\tpl{u}\subset e$. Similarly, we
expose~$\cH\subset\binom{n}{r}$ if we expose all $r$-sets $e\in\cH$.
 
In our algorithm we use the following structure.  A \emph{fan} $\cF(\tpl{u})$ in an
$r$-uniform hypergraph~$\cH$ is a set $\{P_1,\dots,P_t\}$ of tight paths
in~$\cH$ which all have length either~$\ell$ or~$\ell-1$, start in the same
$(r-1)$-tuple $\tpl{u}$, and satisfy the following condition. For any set
$S$ of at least $r/2$ vertices, let $\{P_j\}_{j\in I}$ be the collection of
tight paths in which the set $S$ appears as a consecutive interval. Then
the paths $\{P_j\}_{j\in I}$ also coincide between~$\tpl{u}$ and the
interval~$S$.  The tuple~$\tpl{u}$ is also called the \emph{root}
of~$\cF(\tpl{u})$. Moreover, $\ell$ is the \emph{length} of~$\cF(\tpl{u})$,
and~$t$ its \emph{width}. The set of \emph{leaves}
$L\big(\cF(\tpl{u})\big)$ of~$\cF(\tpl{u})$ is the set of $(r-1)$-tuples
$\tpl{u}'$ such that some path in~$\cF$ ends in~$\tpl{u'}$. For intuition,
observe that in the graph case $r=2$, a fan is simply a rooted tree all of
whose leaves are at distance either $\ell$ or $\ell-1$ from the root. For
$r\ge 3$, a fan is a more complicated structure.

\smallskip

{\bf Idea.}
We shall consecutively build the $\tpl{u}^{(i)}-\tpl{v}^{(i)}$ paths~$P_i$
in the set~$X$,
starting with~$P_1$. The construction of the path~$P_i$ we call
\emph{phase~$i$}, and the strategy in this phase is as follows. We shall
first expose all the hyperedges at~$\tpl{u}^{(i)}$, excluding a set of
`used' vertices~$U$ (like those not in $X$, or in any~$\tpl{u}^{(i')}$
or~$\tpl{v}^{(i')}$). The edges $\{\tpl{u}^{(i)},c\}$ appearing in this
process form possible starting edges for a path connecting $\tpl{u}^{(i)}$
and $\tpl{v}^{(i)}$. For each such (one edge) path~$P$ we next consider the
$(r-1)$-endtuple of~$P$ and expose all edges at this tuple, excluding edges
that were exposed earlier and used vertices (where now we count vertices in
$P$ as used). And so on. In this way we obtain a
(consecutively growing) fan~$\cF(\tpl{u}^{(i)})$ with root~$\tpl{u}^{(i)}$.
While growing this fan we shall also insist that no $j$-tuple of vertices
with $j<r$ is used too often.  We stop when the fan has width
$n^{1-\eps/2}$. We will show that with high probability the fan then has
only constant depth. Then we similarly construct a fan~$\cF(\tpl{v}^{(i)})$
of width $n^{1-\eps/2}$ with root~$\tpl{v}^{(i)}$ (again avoiding used
vertices and exposed edges).

 In a last step, for each leaf~$\tpl{\tilde u}^{(i)}$
of~$\cF(\tpl{u}^{(i)})$ and each leaf $\tpl{\tilde v}^{(i)}$ of
$\cF(\tpl{v}^{(i)})$ we expose all $\tpl{\tilde u}^{(i)}-\tpl{\tilde
  v}^{(i)}$ paths of length $2(r-1)$, avoiding exposed edges. We shall show
that with high probability at least one of these paths appears (and the
fans $\cF(\tpl{u}^{(i)})$ and $\cF(\tpl{v}^{(i)})$ can be constructed), and
hence we have successfully constructed~$P_i$.  We shall also show that, in
phase~$i$ we only exposed much less than a $1/n$ fraction of the $r$-sets
in $X$. Hence it is plausible that we can avoid these exposed $r$-sets in
future phases. We note that this last statement makes use of the fact $r\ge 3$:
our connection algorithm does not work for $2$-graphs.

\begin{proof}[Proof of Lemma~\ref{lem:connect}]

{\sl Setup:} 
  Given $r\ge 3$, $\delta>0$ and $0<\eps<1/(4r)$, we set
  \begin{equation}\label{eq:setxis}   
\xi':=\delta/(48r^2)\,,\quad\xi:=(\xi')^r/(r^2(r-1)!)\quad\text{and}
\quad\eta=\delta/(16r)\,.
  \end{equation}
  Without loss of generality we will assume $|X|=\delta n$: this simplifies our
calculations.

  In the algorithm described below, we maintain various auxiliary sets. We
  have a set $U$ of \emph{used vertices}, which contains all vertices in
  the sets $\tpl{u}^{(i)}$ and $\tpl{v}^{(i)}$, and in previously
  constructed connecting paths. In phase~$i$ we maintain additionally a
  (non-uniform) multihypergraph $U_i$ of \emph{used sets}, which keeps
  track of the number of times we have so far used a vertex, or pair of
  vertices, et cetera, consecutively in some path of the fan currently under
  construction.

  Actually, it will greatly simplify the analysis if any such used set
  can only appear in a unique order on these paths. Hence we choose the
  following setup. We arbitrarily fix an equipartition
  \[X=Y_1\dcup\cdots\dcup Y_{2r}\dcup Y'_1\dcup\cdots\dcup Y'_{2r}\,,\]
  and set $Y:=Y_1\dcup\cdots\dcup Y_{2r}$ and  $Y':=Y'_1\dcup\cdots\dcup Y'_{2r}$. 
  We shall construct the fan~$\cF(\tpl{u}^{(i)})$ with root $\tpl{u}^{(i)}$ in $Y_1,\ldots,Y_{2r}$,
  taking successive levels of the fan from successive sets (in cyclic
  order), and similarly~$\cF(\tpl{v}^{(i)})$ in $Y'_1,\ldots,Y'_{2r}$.

  Further, we maintain an $r$-uniform \emph{expos\'e hypergraph} $H$, which
  keeps track of the $r$-sets which we have exposed. We let $H_i$ be the
  hypergraph with the edges of~$H$ at the beginning of phase~$i$.
  
  We define hypergraphs $D_i^{(1)},\ldots,D_i^{(r-1)}$ of \emph{dangerous
  sets} for phase~$i$ as follows:
  \begin{subequations}
  \begin{align}
    \label{eq:connectr:dangerr-1}
    D_i^{(r-1)}&:=\Big\{\tpl{x}\in\tbinom{X}{r-1}\colon \deg_{H_i}(\tpl{x})\ge\xi
    n\Big\}\,,&&\text{and}\\
    \label{eq:connectr:dangerj}
    D_i^{(j)}&:=\Big\{\tpl{x}\in\tbinom{X}{j}\colon
    \deg_{D_i^{(j+1)}}(\tpl{x})\ge\xi n\Big\}\,,&&\text{$r-2\ge
      j\ge 1$}\,.
  \end{align}
  \end{subequations}
  We will not use any set in any $D_i^{(j)}$ consecutively in a path in the
  fans constructed in phase~$i$.
    
    Given
  two vertex-disjoint $(r-1)$-sets $\tpl{u}$ and $\tpl{v}$, we say that
  the path $(\tpl{u},\tpl{v})$ of length $2r-2$ is \emph{blocked} by the
  expos\'e hypergraph $H$ if any $r$ consecutive vertices of the
  $(2r-2)$-set $\{\tpl{u},\tpl{v}\}$ is in $H$. When constructing the fan
$\cF\big(\tpl{v}^{(i)}\big)$ with root
  $\tpl{v}^{(i)}$, we need to ensure that not too many of its leaves are
  blocked by~$H$ together with too many leaves of the previously
  constructed fan $\cF\big(\tpl{u}^{(i)}\big)$. For this purpose we define
  hypergraphs~$\tilde{D}_i^{(j)}$ of temporarily dangerous sets in
  phase $i$ as follows. We call an $(r-1)$-set $\tpl{y}$ in
  $Y'$ \emph{temporarily dangerous} if there are at least
  $\xi'\big|L\big(\cF(\tpl{u}^{(i)})\big)\big|$ leaves $\tpl{x}$ of
  $\cF(\tpl{u}^{(i)})$ such that $\{\tpl{x},\tpl{y}\}$ is blocked by
  $H_i$. We define
  \begin{subequations}
  \begin{align}
    \label{eq:connectr:dangerir-1}
    \tilde{D}_i^{(r-1)}&:=\Big\{\tpl{y}\in\tbinom{Y'}{r-1}\colon
    \tpl{y}\text{ is temporarily dangerous} \Big\}\,,\quad\text{and} \\
    \label{eq:connectr:dangerij}
    \tilde{D}_i^{(j)}&:=\Big\{\tpl{y}\in\tbinom{Y'}{j}\colon
    \deg_{\tilde{D}_i^{(j+1)}}(\tpl{y})\ge\xi' n\Big\}\,, \quad \text{for }r-2\ge
    j\ge 1\,.
  \end{align}
  \end{subequations}

  Summarising, we do not want  to append a vertex~$c\in X\setminus U$ to
  the end $(r-1)$-tuple $\tpl{a}$ of a path in one of our fans, if for
  $\tpl{a}$ or for any end  $(j-1)$-tuple $\tpl{a}_{j-1}$ of~$\tpl{a}$ with
  $j\in[r-2]$ we have
  \begin{enumerate}[label=\rom]
    \item\label{item:Bad:H} $\{\tpl{a},c\}$ is in~$H$,
    \item\label{item:Bad:D} $\{\tpl{a}_{j-1},c\}$ is an edge of $D_i^{(j)}$ or
      of $\tilde{D}_i^{(j)}$, or
    \item\label{item:Bad:U} $\{\tpl{a}_{j-1},c\}$ has multiplicity greater than
      $\xi^{r-j}n^{(r-1)/2-j(1-\eps)}$ in $U_i$.
  \end{enumerate}
  Hence we define the set~$\Bad(\tpl{a})$ of \emph{bad vertices} for
  $\tpl{a}$ to be the set of vertices in $X\setminus U$ for which at least
  one of these conditions applies.

  \smallskip
  
  {\sl Algorithm:}  
  The desired paths $P_i$ will be constructed using
  Algorithm~\ref{alg:connect}. This algorithm constructs for each
  $i\in[k]$ two fans $\cF(\tpl{u}^{(i)})$ and $\cF(\tpl{v}^{(i)})$, using
  Algorithm~\ref{alg:fan} as a subroutine. 

  \begin{algorithm}[t]
   \caption{Connect each pair $\tpl{u}^{(i)}$, $\tpl{v}^{(i)}$ with a path~$P_i$}
    \label{alg:connect}
    $U:=\bigcup_{i\in[k]}\{\tpl{u}^{(i)},\tpl{v}^{(i)}\}$ ; \quad
    $H:=\emptyset$ \;
   \ForEach{$i\in[k]$}{
     \lnl{step:fanu}construct the fan~$\cF(\tpl{u}^{(i)})$ in
      $Y_1\dcup\ldots\dcup Y_{2r}$ \;
     \lnl{step:fanv}construct the fan~$\cF(\tpl{v}^{(i)})$ in
      $Y'_1\dcup\ldots\dcup Y'_{2r}$ \;
      let $L:=L(\tpl{u}^{(i)})$ be the leaves of~$\cF(\tpl{u}^{(i)})$ \;
      let $L':=L(\tpl{v}^{(i)})$ be the leaves of~$\cF(\tpl{v}^{(i)})$
      reversed \;
      $\cP:=$ all $L-L'$-paths of length~$2r-2$ not blocked by~$H$ \;
      \lnl{step:exposepaths}expose all edges which are in some
      $P\in\cP$ \;
    \uIf{one of these paths~$\tpl{\tilde u}^{(i)},\tpl{\tilde v}^{(i)}$
        appears}{
        $P(\tpl{u^{(i)}}):=\text{the path in $\cF(\tpl{u}^{(i)})$ ending with
        $\tpl{\tilde u}^{(i)}$}$ \;
        $P(\tpl{v^{(i)}}):=\text{reversal of the path in $\cF(\tpl{v}^{(i)})$
        ending with $\tpl{\tilde v}^{(i)}$}$ \;
        $P_i:= P(\tpl{u^{(i)}}),\tpl{\tilde u}^{(i)},\tpl{\tilde v}^{(i)},
        P(\tpl{v^{(i)}})$ \; }
        \lnl{step:failure}\lElse{
        halt with \Failure \;
        }
      \lnl{step:U}$U:=U\cup V(P_i)$ \;
      \lnl{step:Hconnect}\lForEach{$\tpl{x}\in
        L(\tpl{u}^{(i)}),\tpl{y}\in L(\tpl{v}^{(i)})$}{
        $H:=H\cup\binom{\tpl{x}\,\cup\,\tpl{y}}{r}$ \;
      }
    }
 \end{algorithm}

  \begin{algorithm}[t]
    \caption{Construct the fan $\cF(\tpl{u}^{(i)})$}
    \label{alg:fan}
    $\cF(\tpl{u}^{(i)}):=\{\tpl{u}^{(i)}\}$ ; \quad
    $U_i:=\emptyset$ ; \quad
    $t:=1$ \;
    \RepeatForever{}{
      $\cP:=\cF(\tpl{u}^{(i)})$ \;
      \lnl{step:forP}
      \ForEach{path $P\in\cP$}{
        let~$\tpl{a}$ be the end $(r-1)$-tuple of~$P$ \;
        \lnl{step:exposeedges} expose all edges $\{\tpl{a},c\}$ with $c\in
        C':=Y_t\setminus(V(P)\cup U\cup \Bad(\tpl{a}))$ \;
        \lnl{step:C} 
        $C:=\{c \colon \{\tpl{a},c\} \text{ appears in previous step}\}$ \;
        \lnl{step:CheckC} \lIf{\NOT $\delta n^\eps/(16r)\le |C|\le\delta
        n^{\eps}/(2r)$}{halt with
        \Failure \;}
        \lnl{step:F}
        $\cF(\tpl{u}^{(i)}) := \big( \cF(\tpl{u}^{(i)})\setminus\{P\}\big)\cup \big\{ (P,c) \colon c\in C
          \big\}$ \;
         \lnl{step:Qj}
        $\tpl{a}_j:=$ last $j$ vertices of $P$ for $j\in[r-2]$ \;
        \lnl{step:Ui} $U_i:=U_i \cup C \cup \bigcup_{c\in
        C}\big\{\{\tpl{a}_j,c\}\colon j\in[r-2]\big\}$ \;
        \lnl{step:Hfan}
        $H:=H\cup\big\{(\tpl{a},c)\colon c\in C'\big\}$
        \;
		\lnl{step:whileF} \lIf{$|\cF(\tpl{u}^{(i)})|\ge
		n^{(r-1)/2-\eps/2}$}{\Return\ \;} }
      $t:=(t \mod 2r) + 1$ \;
   }
 \end{algorithm}
    
  It is clear that the running time (whether the algorithm succeeds or
  fails) is polynomial: Steps~\ref{step:CheckC} and~\ref{step:whileF}
  guarantee that in one call, Algorithm~\ref{alg:fan} runs at most
  $n^{(r-1)/2}$ times through its repeat loop.  Our analysis will
  show that a.a.s.\ the algorithm indeed succeeds.
  
  \smallskip

Before we proceed with the analysis, let us remind the reader that 
$H$ denotes the already exposed hyperedges that appeared so far, $H_i$ consists
 of the hyperedges of $H$ before the start of phase $i$, 
  $U$ is the set of already used vertices 
 and $U_i$ is the auxiliary multihypergraph which is maintained through phase $i$ and 
 records those  $j$-tuples ($j\in[r-1]$)
  that were used for constructing the fan $\cF(\tpl{u}^{(i)})$ ($\cF(\tpl{v}^{(i)})$ resp.).

  \smallskip

  {\sl Analysis:} First, we claim that the algorithm is valid in that it
  does not try to expose any $r$-set twice. To see this, we need to check
  that at steps~\ref{step:exposepaths} and~\ref{step:exposeedges}, we do
  not attempt to re-expose an already exposed $r$-set. Since we do not
  expose any $r$-set in $H$ at either step (by the definition of
  $\Bad(\tpl{a})$), it is enough to check that after either step, all
  exposed $r$-sets are added to $H$ before the next visit to either step.
  This takes place in steps~\ref{step:Hconnect} and~\ref{step:Hfan}.
  
  In order to show that the algorithm succeeds, we need to
  show that the following hold with sufficiently high probability for
  each $i\in[k]$. \begin{enumerate}[label=\itm{A\arabic{*}}, start=1]
    \item\label{item:connect:A2} Algorithm~\ref{alg:fan} successfully builds
    the fans~$\cF(\tpl{u}^{(i)})$ and~$\cF(\tpl{v}^{(i)})$, that is, the
      condition in step~\ref{step:whileF} eventually becomes true, and the
      condition in step~\ref{step:CheckC} never becomes true.
    \item\label{item:connect:A3} If this is the case, then
      Algorithm~\ref{alg:connect} successfully constructs~$P_i$, that is,
      one of the paths exposed in step~\ref{step:exposepaths} appears.
    \item\label{item:connect:A4} If this is the case, then~$P_i$ is
      of length at most~$s=\frac{r-1}{\eps}$, that is, the
      fans~$\cF(\tpl{u}^{(i)})$ and~$\cF(\tpl{v}^{(i)})$ have length at most~$s/2$.
  \end{enumerate}

  It is straightforward to see that~\ref{item:connect:A4} holds. Indeed, if
  Algorithm~\ref{alg:fan} succeeds in step~$i$, then in the last repetition
  of the for-loop creating $\cF(\tpl{u}^{(i)})$, the width of
  $\cF(\tpl{u}^{(i)})$ finally exceeds $n^{(r-1)/2-\eps/2}$. Since by
  step~\ref{step:CheckC} at most
  $|C|\le \delta n^\eps/(2r)<n^{(r-1)/2-\eps/2}$ paths are added to
  $\cF(\tpl{u}^{(i)})$ in this last for-loop (and the same holds for
  $\cF(\tpl{v}^{(i)})$), we obtain for the width of $\cF(\tpl{u}^{(i)})$
  and $\cF(\tpl{v}^{(i)})$ (which equals the number of their leaves) that
  \begin{equation}\label{eq:connect:L}
    n^{(r-1)/2-\eps/2}\le
    \big|L(\tpl{u}^{(i)})\big|,\big|L(\tpl{v}^{(i)})\big|\le
    2n^{(r-1)/2-\eps/2}\,.
  \end{equation}
  Now observe that by step~\ref{step:CheckC} the fan $\cF(\tpl{u}^{(i)})$
  (and similarly~$\cF(\tpl{v}^{(i)})$) has width at least $\big(\delta
  n^{\eps}/(16r)\big)^{s_i}$, where $s_i$ is the length
  of~$\cF(\tpl{u}^{(i)})$. For $s_i\ge(r-1)/(2\eps)$ this would imply
  $\big|L(\tpl{u}^{(i)})\big|\ge \big(\delta
  n^{\eps}/(16r)\big)^{(r-1)/(2\eps)}>n^{(r-1)/2-\eps/2}$,
  contradicting~\eqref{eq:connect:L}.
  Hence we have~\ref{item:connect:A4}.
   
  For proving~\ref{item:connect:A2} and~\ref{item:connect:A3}, we first
  show bounds on various quantities during the running of the algorithm.
  For a set $\tpl{a}$ in the multiset $U_i$ with $i\in[k]$, we write
  $\mult_{U_i}(\tpl{a})$ for the multiplicity of~$\tpl{a}$ in~$U_i$.

  \begin{claim}\label{cl:connect:F}
    If phase~$i$ and all phases before succeed, then the following hold throughout phase~$i$.
    \begin{enumerate}[label=\abc]
      \item\label{item:connect:U} $|U|\le k\big(s+2(r-1)\big)\le 2kr/\eps$.
      \item\label{item:connect:Ui} For each $j\in[r-1]$ and each $j$-set
        $\tpl{a}\in U_i$ we have
        \[\mult_{U_i}(\tpl{a})\le\xi^{r-j}n^{((r-1)/2)-j(1-\eps)}+1\,.\]
     \item\label{item:connect:Uij} For each $j\in[r-1]$ and each
        $(j-1)$-set $\tpl{a}$ in $[n]$, for all but $\xi n$ vertices
        $c\in X$ we have \[\mult_{U_i}(\{\tpl{a},c\})\le\xi^{r-j}
        n^{((r-1)/2)-j(1-\eps)}\,.\]
      \item\label{item:connect:H} $e(H_{i+1})\le 2^{2r+1}(i+1)n^{r-1-\eps/2}$.
      \item\label{item:connect:J} At step~\ref{step:exposeedges} in
      Algorithm~\ref{alg:fan}, we have $|Y_t\setminus(V(P)\cup U\cup  \Bad(\tpl{a}))|\ge\delta n/(8r)$.
   \end{enumerate}
  \end{claim}
  Observe that for $j\ge r/2$ Claim~\ref{cl:connect:F}\ref{item:connect:Ui}
  implies that we always have $\mult_{U_i}(\tpl{a})\le 1$ for any $j$-tuple
  used in any $\cF(\tpl{u}^{(i)})$ or $\cF(\tpl{v}^{(i)})$. This shows
  that $\cF(\tpl{u}^{(i)})$ and $\cF(\tpl{v}^{(i)})$ are indeed fans, as we claim.

  \begin{claimproof}[Proof of Claim~\ref{cl:connect:F}]
    We first prove~\ref{item:connect:U}.
    The set $U$ contains the $2k(r-1)$ vertices of the $k$ pairs of
    $(r-1)$-tuples which we wish to connect, together with all the vertices of
    the paths thus far constructed. Since by~\ref{item:connect:A4} these
    paths are of length at most~$s$, it follows that $|U|\le 2k(r-1)+(i-1)s\le
    k\big(s+2(r-1)\big)$.

    \smallskip

    To see that~\ref{item:connect:Ui} holds, observe that $j$-sets are
    added to~$U_i$ only at step~\ref{step:Ui}, and at this point the sets
    added are distinct: two sets either contain different members of $Y_t$,
    or they are of different sizes. Moreover, they are added only if their
    multiplicity in $U_i$ is at most $\xi^{r-j} n^{(r-1)/2-j(1-\eps)}$
    by~\ref{item:Bad:U} in the definition of $\Bad(\tpl{a})$.

    \smallskip

    For~\ref{item:connect:Uij} we proceed by induction on~$j$.  First
    consider the case~$j=1$. Observe that $c\in X$ is added to~$U_i$ in
    step~\ref{step:Ui} only if it is added at the end of a path~$P$. Since
    step~\ref{step:CheckC} guarantees that each fan grows by a factor of at
    least~$2$ in each iteration, we have
    \begin{equation*}
      \sum_{c\in X} \mult_{U_i}(c)
      \le 2\big(|L(\tpl{u}^{(i)})|+L(\tpl{v}^{(i)})|\big)
      \leByRef{eq:connect:L} 4n^{(r-1)/2-\eps/2}<\xi^rn^{(r-1)/2}\,.
    \end{equation*}
   We conclude that there are at most
    \begin{equation*}
      \frac{\xi^rn^{(r-1)/2}}
      {\xi^{r-1}n^{((r-1)/2)-1+\eps}}
      = \xi n^{1-\eps}
    \end{equation*}
    vertices $c\in X$ with $\mult_{U_i}(c)>\xi^{r-1}n^{((r-1)/2)-1+\eps}$.

    Now assume that~\ref{item:connect:Uij} holds for $j-1$ and
    let~$\tpl{a}$ be a $(j-1)$-set in $[n]$. Similarly as before, for $c\in X$ the set
    $\{\tpl{a},c\}$ is in $U_i$ with multiplicity equal to the number of
    times that~$\tpl{a}$ has appeared as the end of a path~$P$ in one of
    the two fans constructed in this phase
    and the path $(P,c)$ was subsequently added to the fan in
    step~\ref{step:F}. Since we did not previously halt in
    step~\ref{step:CheckC}, for any $P$ there are at most $\delta
    n^\eps/(2r)\le\frac12 n^{\eps}$ vertices $c\in X$ such that $(P,c)$ is added in this way.
    Thus we have
    \begin{equation}\label{eq:connect:mult}
      \sum_{c\in X} \mult_{U_i}(\tpl{a},c)\le \mult_{U_i}(\tpl{a})\cdot
      \tfrac{1}{2} n^{\eps}\,.
    \end{equation}
    By~\ref{item:connect:Ui} we know in addition that
    \[\mult_{U_i}(\tpl{a})\le\xi^{r-j+1}n^{((r-1)/2)-(j-1)(1-\eps)}+1\,.\] Note
    that if this bound is less than~$2$ then
    \eqref{eq:connect:mult} directly implies that there are at most~$\xi n$
    vertices~$c$ with $\mult_{U_i}(\tpl{a},c)\ge 1$ and we are done. Hence
    we may assume
    $\mult_{U_i}(\tpl{a})\le2\xi^{r-j+1}n^{((r-1)/2)-(j-1)(1-\eps)}$. This
    together with~\eqref{eq:connect:mult} also implies that there are at
    most
    \begin{equation*}
      \frac{2\xi^{r-j+1}n^{((r-1)/2)-(j-1)(1-\eps)} \cdot \frac12
        n^{\eps}}{\xi^{r-j}n^{((r-1)/2)-j(1-\eps)}}
      =\xi n
    \end{equation*}
    vertices $c\in X$ with
    $\mult_{U_i}(\tpl{a},c)\ge\xi^{r-j}n^{((r-1)/2)-j(1-\eps)}$, as desired.

    \smallskip

    For the remaining parts of the claim, we proceed by induction on the
    phase~$i\in[k]$.  So assume that the claim holds at the end of the
    $(i-1)$st phase.

    We next prove~\ref{item:connect:H}.  At the end of phase~$i$, the
    hypergraph $H$ contains all the $r$-sets which it had at the end of
    phase $i-1$, together with all those added in phase~$i$.  Now consider
    the construction of one fan in phase~$i$, say of
    $\cF(\tpl{u}^{(i)})$. Since we did not halt in
    step~\ref{step:CheckC}, the width of the fan grows exponentially, more than
doubling at each step. Thus we can bound the total number of iterations of
    the for-loop by the number $|L(\tpl{u}^{(i)})|$ of leaves of this fan
    (cf.\ step~\ref{step:F}).
    In each of these iterations, we exposed
    $|Y_t\setminus(P\cup U\cup \Bad(\tpl{a}))|<n$ of the $r$-sets and added
    them to~$H$. Hence, while constructing $\cF(\tpl{u}^{(i)})$ (and
    similarly for $\cF(\tpl{v}^{(i)})$), we added
    at most $|L(\tpl{u}^{(i)})|n$ new $r$-sets to~$H$. 
    The only other step where we add $r$-tuples to~$H$ is
    step~\ref{step:Hconnect}. In this step, for each pair of leaves of
    $\cF(\tpl{u}^{(i)})$ and $\cF(\tpl{v}^{(i)})$, we add
    $\binom{2r-2}{r}$ new $r$-sets to $H$. Using the induction hypothesis
    we thus conclude that at the end of phase~$i$ we have
    \begin{equation*}\begin{split}
      e(H_{i+1}) &\le e(H_i) + \big(|L(\tpl{u}^{(i)})|+|L(\tpl{v}^{(i)})|\big)n
      + \tbinom{2r-2}{r} |L(\tpl{u}^{(i)})|\cdot|L(\tpl{v}^{(i)})| \\
      & \leByRef{eq:connect:L}
      2^{2r+1} i\cdot n^{r-1-\eps/2}
      + 4 n^{(r+1)/2-\eps/2} \cdot n
      + \tbinom{2r-2}{r} 4n^{r-1-\eps} \\
      & \le 2^{2r+1}(i+1)n^{r-1-\eps/2} \,,
    \end{split}\end{equation*}
    where for the final inequality we use the fact that $(r+1)/2\le r-1$, which
    holds since $r\ge 3$. This is the only step in the analysis where we use
    $r\ge 3$, but this analysis is reasonably tight: the algorithm does fail
for $r=2$.

    \smallskip
      
    Last we prove~\ref{item:connect:J}, for which we additionally proceed
    by induction on the number~$f$ of iterations through the for-loop of
    Algorithm~\ref{alg:fan} done in the $i$th phase so far.  So we assume
    that the claim holds at the end of the $(i-1)$st phase and after
    $f-1$ iterations.

    Let $P$ be the path considered in iteration~$f$ of this for-loop,
    and $\tpl{a}$ the $(r-1)$-tuple ending $P$.  We would like to estimate
    the size of $\Bad(\tpl{a})\cap Y_t$.  Keep in mind in the following analysis
    that for $j\in[r-1]$ the hypergraph~$D_i^j$ does not change during
    phase~$i$, by definition. Similarly, $\tilde D_i^j$ does not change
    once the fan $\cF(\tpl{u}_i)$ is constructed.

    Now let us first assess the effect of~\ref{item:Bad:H} of the
    definition of $\Bad(\tpl{a})$.  Since~$\tpl{a}$ is the end of a path
    constructed by Algorithm~\ref{alg:fan}, step~\ref{step:exposeedges}
    implies that the last vertex~$b$ of~$\tpl{a}$ is not contained in
    $\Bad(\tpl{b})$ where~$\tpl{b}$ is the end $(r-1)$-tuple of
    $P-b$. From~\ref{item:Bad:D} in the definition of $\Bad(\tpl{b})$ we
    conclude that $\tpl{a}\notin D_i^{(r-1)}$. Thus, by the definition of
    $D_i^{(r-1)}$ in~\eqref{eq:connectr:dangerr-1}, the number of edges in
    $H_i$ containing $\tpl{a}$ is smaller than $\xi n$.
    
    But how many edges $\{\tpl{a},c\}$ with $c\in Y_t$ did phase~$i$
    add to $H$ so
    far?  By~\ref{item:connect:Ui} the set~$\tpl{a}$ has multiplicity at
    most $\xi n^{(r-1)(2\eps-1)/2}+1<2$ in~$U_i$. It follows that since the
    start of phase~$i$ only one edge containing~$\tpl{a}$ was added to~$H$
    in step~\ref{step:Hfan}: the end $r$-tuple $\tpl{a}_r$ of~$P$.
    However, since~$\tpl{a}_r$ contains no vertices of~$Y_t$ because the
    algorithm takes successive levels of the fan in successive~$Y_{t'}$
    (or~$Y'_{t'}$), we conclude that the current phase did not add any
    additional edges $\{\tpl{a},c\}$ to~$H$ with $c\in Y_t$.

    Now let us estimate the number of vertices~$c\in Y_t$
    which~\ref{item:Bad:D} of the definition of $\Bad(\tpl{a})$ forbids.
    First, we need to consider the case $j=1$, and show that the number of
    vertices in $D_i^{(1)}$ is at most
    $\xi n$. Suppose not, and observe that by
    definition~\eqref{eq:connectr:dangerj}, each vertex in $D_i^{(1)}$ extends to
    at least $\xi n$ pairs in $D_i^{(2)}$, and so on, where at the final step
    each constructed member of $D_i^{(r-1)}$ extends to at least $\xi n$
    members of $H_i$. We can construct any given member of $H_i$ in at most $r!$
    ways, so we conclude that $e(H_i)\ge(\xi n)^r/r!$, which (for
    sufficiently large $n$) contradicts part~\ref{item:connect:H}.

    Next, again for the case $j=1$, we need to show that further there are at
    most $\xi' n$ vertices in $\tilde{D}_i^{(1)}$. Again, suppose not: then
    as above this implies that the number of pairs of $(r-1)$-tuples
    $(\tpl{x},\tpl{y})$ with $\tpl{x}\in L\big(\tpl{u}^{(i)}\big)$ and
    $\tpl{y}$ contained in $Y'_1\cup\ldots\cup Y'_{2r}$ is at least
    \begin{equation}\label{eq:connect:tempdang}
     \xi'\Big|L\big(\tpl{u}^{(i)}
     \big)\Big|\cdot(\xi'n)^ { r-1 } /(r-1)! \geByRef{eq:setxis}r^2\xi
     n^{r-1}\Big|L\big(\tpl{u}^{(i)}\big)\Big|\,.
    \end{equation}
    However, by construction of $\cF(\tpl{u}^{(i)})$, for each $j\in[r-1]$ and
    each $\tpl{x}\in L\big(\tpl{u}^{(i)}\big)$, we have the property that the
    last $j$ vertices of $\tpl{x}$ are not in $D_i^{(j)}$. We claim that this
    implies that the number of $(r-1)$-tuples $\tpl{y}$ contained in
    $Y'_1\cup\ldots\cup Y'_{2r}$ such that $(\tpl{x},\tpl{y})$ is blocked by
    $H_i$, is at most $(r-1)^2\xi n^{r-1}$, which is a contradiction
    to~\eqref{eq:connect:tempdang}. To see this, consider the following property
    P of tuples $\tpl{y}$. For each $r-1\ge j\ge 1$ and each $1\le k\le r-j$, the
    tuple consisting of the last $j$ vertices of $\tpl{x}$ followed by the first
    $k$ vertices of $\tpl{y}$ is not in $D_i^{(j+k)}$ (if $j+k<r$) and not in
    $H_i$ (if $j+k=r$). If $\tpl{y}$ has property P, then clearly the pair
    $(\tpl{x},\tpl{y})$ is not blocked by $H_i$. On the other hand, if $\tpl{y}$
    does not have P, then there is a smallest $k$ for which P fails. By
    definition of the sets $D_i^{(j+k)}$, for a fixed~$j$ given the first $k-1$ vertices of
    $\tpl{y}$ there are at most~$\xi n$ choices for the $k$-th vertex of
    $\tpl{y}$. Hence, in total, given the first $k-1$ vertices of
    $\tpl{y}$ there are at most $(r-1)\xi n$ choices for the $k$-th vertex of
    $\tpl{y}$. Thus the number of $(r-1)$-tuples $\tpl{y}$ which do not have
    P is at most $(r-1)^2\xi n^{r-1}$ as desired.

    Now given $2\le j\le r-2$, let $\tpl{a}_{j-1}$ be the set consisting of the
    last $j-1$ vertices of $\tpl{a}$. By construction, $\tpl{a}_{j-1}$ is in
    neither $D_i^{(j-1)}$ nor in $\tilde{D}_i^{(j-1)}$.  It follows from
    the definition of these sets in~\eqref{eq:connectr:dangerj}
    and~\eqref{eq:connectr:dangerij} that there are at most $\xi n$
    vertices $c$ such that $\{\tpl{a}_{j-1},c\}\in D_i^{(j)}$, and at most $\xi'
    n$ such that $\{\tpl{a}_{j-1},c\}\in \tilde{D}_i^{(j)}$. Together with
    the case $j=1$, this gives at most $(r-2)(\xi+\xi')n$ forbidden vertices
    $c\in Y_t$.

    Finally, for~\ref{item:Bad:U}, observe that by
    part~\ref{item:connect:Ui}, for each $j\in[r-2]$ there are at most $\xi
    n$ vertices $c\in X$ with
    $\mult_{U_i}\big(\{\tpl{a}_{j-1},c\}\big)>\xi^{r-j} n^{\frac{r-1}2-j(1-\eps)}$.

    Hence, in total, $\Bad(\tpl{a})\cap Y_t$ contains at most
   \[\xi n + (r-2)(\xi+\xi')n+(r-2)\xi n \leByRef{eq:setxis} \frac{\delta n}{4r}\]
   vertices. Moreover, it follows from~\ref{item:connect:A4} that $|P|\le
   r/\eps$, and from~\ref{item:connect:U} that $|U|\le 2kr/\eps$. Since we
   have $|Y_t|=\delta n/(2r)$, we conclude that
   \[|Y_t\setminus(P\cup U\cup  \Bad(\tpl{a}))|\ge \frac{\delta n}{2r} -
   \frac{r}{\eps} - \frac{2kr}{\eps} -\frac{\delta n}{4r}
   \ge\frac{\delta n}{8r}\]
 \end{claimproof}
  
 Now we can use a Chernoff bound to show that a.a.s.\
 Algorithm~\ref{alg:fan} does not fail in step~\ref{step:CheckC}.
 
  \begin{claim}\label{cl:connect:checkC}
    At any given visit to step~\ref{step:CheckC}, Algorithm~\ref{alg:fan} halts
    with probability at most $2\exp\big(-\delta n^\eps/(96r)\big)$.
  \end{claim}
  \begin{claimproof}
    By Claim~\ref{cl:connect:F}\ref{item:connect:J}, we have \[\delta
    n/(8r)\le |Y_t\setminus(P\cup U\cup \Bad(\tpl{a}))|\le |Y_t|=\delta
    n/(4r)\,.\] Since $C$ is a $p$-random subset of $Y_t\setminus(P\cup U\cup
    \Bad(\tpl{a}))$ with $p=n^{-1+\eps}$, we obtain
    $\delta n^{\eps}/(8r)\le\Exp|C|\le \delta n^\eps/(4r)$.
    Using the Chernoff bound~\eqref{eq:chernoff} with $\gamma=1/2$, we
    conclude that
    $\delta n^{\eps}/(16r)\le |C|\le \delta
    n^\eps/(2r)$
    with probability at least $1-2\exp\big(-\delta n^\eps/(96r)\big)$.
  \end{claimproof}

  We would like to show that also a.a.s.\ Algorithm~\ref{alg:connect} does
  not fail in step~\ref{step:failure}. Since the events considered in this
  step are not mutually independent, we use Janson's inequality for this
  purpose.
  
  \begin{claim}\label{cl:connect:failure}
    At any given visit to step~\ref{step:failure}, Algorithm~\ref{alg:connect}
    halts with probability at most $\exp(-n^{(r-2)\eps}/4)$.
  \end{claim}
  \begin{claimproof}
    Let~$\cE=\bigcup\cP$ be the family of $r$-sets exposed in
    step~\ref{step:exposepaths} in this iteration of the foreach-loop.  For
    each~$P\in\cP$ let $I_P$ be the indicator variable for the event that
    the path~$P$ appears, which occurs with probability $\tilde
    p=p^{r-1}$. Then $X=\sum_{P\in\cP} I_P$ is the random variable counting
    the number of $L-L'$-paths appearing in this iteration. We would like
    to use Janson's inequality~\eqref{eq:Janson} to show that $X>0$ with
    high probability, in which case Algorithm~\ref{alg:connect} does not
    halt in step~\ref{step:failure}.

    To this end we first bound $\Exp X$, for which we need to estimate
    $|\cP|$. Firstly, since $\cF(\tpl{u}^{(i)})$ and $\cF(\tpl{v}^{(i)})$
    are disjoint fans, no vertex is in a leaf both of
    $\cF(\tpl{u}^{(i)})$ and of $\cF(\tpl{v}^{(i)})$, and in particular
$L(\tpl{u}^{(i)})$ and $L(\tpl{v}^{(i)})$ are disjoint. Now let
    $\tpl{\tilde v}$ be any $(r-1)$-tuple in $L(\tpl{v}^{(i)})$. By
    construction $\tpl{\tilde v}$ is not in $\tilde{D}_i^{(r-1)}$ (see
    step~\ref{step:exposeedges} and the definition of $\Bad(\tpl{a})$).
    By~\eqref{eq:connectr:dangerir-1} the path $(\tpl{\tilde u},\tpl{\tilde
      v})$ therefore is blocked by~$H$ for at most
    $\xi'\big|L(\tpl{u}^{(i)})\big|$ tuples $\tpl{\tilde{u}}\in L(\tpl{u}^{(i)})$.
    Thus we have
    \begin{equation*}
      |\cP| \ge
      |L(\tpl{v}^{(i)})\big|\cdot
      (1-\xi')|L(\tpl{u}^{(i)})\big|
      \geByRef{eq:connect:L} (1-\xi')n^{r-1-\eps}\,,
   \end{equation*} 
   which gives
   \begin{equation}\label{eq:connect:E}
     \Exp X=|\cP|\tilde p \ge (1-\xi')n^{r-1-\eps}n^{(\eps-1)(r-1)}=
     (1-\xi')n^{(r-2)\eps}\,.
   \end{equation}
   
   Next we would like to estimate $\Exp(I_PI_{P'})$ for two distinct paths
   $P=(\tpl{\tilde u},\tpl{\tilde v})$ and $P'=(\tpl{\tilde u}',\tpl{\tilde
     v}')$ which share at least one edge.  If $P$ and $P'$ are distinct
paths sharing at least one edge, then in
   particular, either $\tpl{\tilde u}$ and $\tpl{\tilde u}'$ have the same
   end $r/2$-tuple, or $\tpl{\tilde v}$ and $\tpl{\tilde v}'$ have the same
   start $r/2$-tuple. Without loss of generality assume the former and
   suppose that $\tpl{\tilde v}$ and $\tpl{\tilde v}'$ match in the start
   $j$-tuple, but not in the $(j+1)$st vertex. Clearly $1\le j$,
   and since $\cF(\tpl{u}^{(i)})$ and $\cF(\tpl{v}^{(i)})$ are fans we have
   $\tpl{\tilde u}=\tpl{\tilde u'}$ and $j<r/2$. Hence~$P$ and~$P'$ share
   precisely an interval of length $r-1+j$, and thus~$j$ edges. Therefore
   $\Exp(I_PI_{P'})\le p^{2r-2-j}$.

   In addition, the discussion above shows that for a fixed path
   $P=(\tpl{\tilde u},\tpl{\tilde v})$, the number $N_{P,j}$ of
   paths~$P'=(\tpl{\tilde u}',\tpl{\tilde v}')$ such that $P$ and $P'$
   share $j$ edges, is at most the number of choices of a leaf $\tpl{\tilde
     v}'\in L(\tpl{v}^{(i)})$ such that $\tpl{\tilde v}'$ only has the end
   $(r-1-j)$-tuple $\tpl{v}$ different from $\tpl{\tilde v}$, plus the
   number of choices of a leaf $\tpl{\tilde u}'\in L(\tpl{u}^{(i)})$ such
   that $\tpl{\tilde u}'$ only has the start $(r-1-j)$-tuple $\tpl{u}$
   different from $\tpl{\tilde u}$.
   By Claim~\ref{cl:connect:F}\ref{item:connect:Ui} the start $j$-tuple of
   $\tpl{\tilde v}$ and the end $j$-tuple of $\tpl{\tilde u}$ have multiplicity
   in~$U_i$ at most $n^{(r-1)/2-j(1-\eps)}+1$. By
   step~\ref{step:Ui} this implies that there are at most
   $n^{(r-1)/2-j(1-\eps)}$ choices for $\tpl{u}$ and for $\tpl{v}$, and
   hence $N_{P,j}\le 2 n^{(r-1)/2-j(1-\eps)}$.

   With this we are ready to estimate
   \begin{equation*}
     \Delta=\sum_{P\neq P',P\cap P'=\emptyset}\Exp(I_PI_{P'})
     =\sum_P\sum_{1\le j<r/2}\Big(\sum_{|P'\cap P|=j} \Exp(I_PI_{P'})\Big),
   \end{equation*}
   where $P,P'\in\cP$.  We have
   \begin{equation*}
     \begin{split}
     \Delta&\le\sum_P\sum_{1\le j<r/2} N_{P,j} \cdot p^{2r-2-j} \\
     &\le|L(\tpl{u}^{(i)})||L(\tpl{v}^{(i)})| \sum_{1\le j<r/2} 2
     n^{(r-1)/2-j(1-\eps)} p^{2r-2-j},
     \end{split}
   \end{equation*}
   which, by~\eqref{eq:connect:L}, is at most
   \begin{multline*}
     n^{(r-1)-\eps} \sum_{1\le j<r/2} 2
     n^{(r-1)/2-j(1-\eps)} n^{(\eps-1)(2r-2-j)},\\
     \leq n^{(r-1)-\eps} \cdot r \cdot
     n^{-\frac32(r-1)+2\eps(r-1)}<1.
   \end{multline*}
   Hence, inequalities~\eqref{eq:Janson} and~\eqref{eq:connect:E}
   imply that $\Prob(X=0)\le \exp(\Delta-\Exp X)\le \exp(-n^{(r-2)\eps}/4)$,
   and thus Algorithm~\ref{alg:connect} fails with at most this probability
   in this visit to step~\ref{step:failure}
 \end{claimproof}

  Since Algorithm~\ref{alg:connect} visits step~\ref{step:failure} at most
  $k\le n$ times, we can use Claim~\ref{cl:connect:failure} and a union bound to
  infer that~\ref{item:connect:A3} holds with probability at least
  $1-n\cdot\exp\big(-n^{(r-2)\eps}/4\big)\ge 1-\frac12\exp\big(-\delta
  n^\eps/(100 r)\big)$.
  Similarly, step~\ref{step:CheckC} of Algorithm~\ref{alg:fan} is called at
  most once per leaf in any of the at most $2k$ constructed fans, which is
  at most $2k\cdot 2n^{(r-1)/2-\eps/2}\le n^r$ times
  by~\eqref{eq:connect:L}. It follows from Claim~\ref{cl:connect:checkC}
  that~\ref{item:connect:A2} holds with probability at least
  $1-n^r \cdot 2\exp\big(-\delta n^\eps/(96r)\big)\ge 1-
  \frac12\exp(-\delta n^\eps/(100r)\big)$.

  Summarising, we showed that Algorithm~\ref{alg:connect} constructs
  the~$k$ desired tight paths of length at most~$\ell$ with probability at
  least $1-\exp(-\delta n^\eps/(100r)\big)$.
\end{proof}

\section{Proof of the reservoir lemma}
\label{sec:reserve}

In this section we prove Lemma~\ref{lem:reserve}.

\begin{proof}[Proof of Lemma~\ref{lem:reserve}]
  Choose $\ell:=\big\lceil 1/\big(2(r-1)\eps\big) \big\rceil+2$.  Our
  strategy will be as follows. We will start by defining an auxiliary
  $r$-uniform hypergraph~$\Drl$ with $2(r-1)(2\ell-1)+1$ vertices and as
  many edges, which implies
  \begin{equation}\label{eq:do:D}
    d^{(1)}(\Drl)=1+\frac{1}{2(r-1)(2\ell-1)} \le 1 + \eps \,.
  \end{equation}
  After defining~$\Drl$ we shall construct a graph~$\cH^*$
  which satisfies~\ref{lem:reserve:2} and~\ref{lem:reserve:3} and is such that
  $\Drl\subset\cH^*$ and $\Drl$~has maximum $1$-density among all
  subhypergraphs of~$\cH^*$.

  \smallskip

  The vertex set of $\Drl$ is
  \begin{equation*}
    V(\Drl):=U\cup V\cup\bigcup_{i\in[\ell-1]} A_i \cup \bigcup_{i\in[\ell-2]}
    B_i\,.
  \end{equation*}
  where $U:=(u_1,\dots,u_{r-1},w^*,u_r,\dots,u_{2(r-1)})$,
  $V:=(v'_1,\dots,v'_{2(r-1)})$, $A_i:=(a^{(i)}_1,\dots a^{(i)}_{2(r-1)})$ for
  $i\in[\ell-1]$, and $B_i:=(b^{(i)}_1,\dots b^{(i)}_{2(r-1)})$ for
  $i\in[\ell-2]$ are ordered sets of vertices.  The edge set of $\Drl$ contains
  exactly the edges of the tight paths determined by~$U$, by~$V$, by~$A_i$ for each
  $i\in[\ell-1]$, by~$B_i$ for each $i\in[\ell-2]$, as well as
  by the following vertex sequences:
  \begin{align*}
    & \tilde U_A:= (u_{1},\dots,u_{r-1},a^{(1)}_{r-1},\dots,a^{(1)}_{1})  \,, \\
    & \tilde V_A:=(a^{(\ell-1)}_{2(r-1)},\dots,a^{(\ell-1)}_{r},v'_{r},\dots,v'_{2(r-1)})  \,, \\
    & \tilde U_B:=(u_{2(r-1)},\dots,u_{r},b^{(1)}_{r-1},\dots,b^{(1)}_{1})  \,, \\
    & \tilde V_B:=(b^{(\ell-2)}_{2(r-1)},\dots,b^{(\ell-2)}_{r},v'_{r-1},\dots,v'_{1})  \,,
  \end{align*}
  and
  \begin{align*}
    & \tilde A_{i,i+1}:=(a^{(i)}_{2(r-1)},\dots,a^{(i)}_{r},a^{(i+1)}_{r-1},\dots,a^{(i+1)}_{1}) &&
    \quad\text{for all $i\in[\ell-2]$} \,,\\
    & \tilde B_{i,i+1}:=(b^{(i)}_{2(r-1)},\dots,b^{(i)}_{r},b^{(i+1)}_{r-1},\dots,b^{(i+1)}_{1}) &&
    \quad\text{for all $i\in[\ell-3]$} \,.
  \end{align*}
  It is not difficult to check that~$\Drl$ has exactly $2(r-1)(2\ell-1)+1$
  vertices and edges as claimed.

  \begin{figure}[t]
    \begin{center}
      \psfrag{U}{\scalebox{1.3}{$U$}}
      \psfrag{V}{\scalebox{1.3}{$V$}}
      \psfrag{A1}{\scalebox{1.3}{$A_1$}}
      \psfrag{A2}{\scalebox{1.3}{$A_2$}}
      \psfrag{B1}{\scalebox{1.3}{$B_1$}}
      \psfrag{w}{\scalebox{0.8}{$w^*$}}
      \includegraphics[scale=.85]{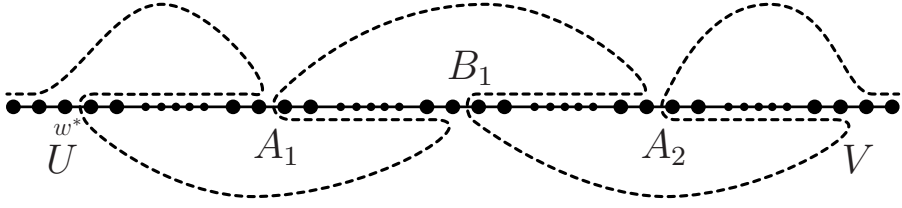}
      \caption{$\cH^*$ for $r=3$ and $\ell=3$. The vertices of~$\Drl$ are drawn
        bigger than the vertices newly inserted in~$\cH$. The continuous line indicates the tight
        Hamilton path in~$\cH^*$ from~\eqref{eq:reserve:path}, the dashed line the tight
        Hamilton path in $\cH^*-w^*$.
      }\label{fig:F}
    \end{center}
  \end{figure}

  In order to obtain~$\cH^*$ from~$\Drl$ we first let $v_i:=v'_{(r-1)+i}$ for each
  $i\in[r-1]$. Then we insert \[k:=3(r-1)^2(2\ell-1)\] new vertices `between' each
  of the following pairs of vertex sets in $\Drl$: $U$ and~$A_1$, $A_{\ell-1}$
  and~$V$, $A_i$ and~$B_i$ for each $i\in[\ell-2]$, $B_i$ and~$A_{i+1}$ for each
  $i\in[\ell-2]$. We let $I(X,Y)$ denote the ordered set of vertices inserted
  `between' the sets~$X$ and~$Y$ in this process (where we choose any ordering).
  In addition, we add to this graph the tight Hamilton path 
  \begin{multline}\label{eq:reserve:path}
    U, I(U,A_1), A_1, I(A_1,B_1), B_1, I(B_1,A_2), A_2,\dots \\
    \dots, B_{\ell-2}, I(B_{\ell-2},A_{\ell-1}),A_{\ell-1},I(A_{\ell-1},V),V
  \end{multline}
  running from $\tpl{u}$ to $\tpl{v}$ (which uses some edges already
  present in~$\Drl$). The resulting hypergraph is~$\cH^*$ (see
  also Figure~\ref{fig:F}).  

  By construction $v(\cH^*)=2(r-1)(2\ell-1)+1+2k(\ell-1)$
  which by definition of~$k$ is smaller than $16(r-1)^2\ell^2\le
  16\eps^{-2}$, and $e(\cH^*)=2(r-1)(2\ell-1)+1+2(k+r-1)(\ell-1)$. Since
  $\ell>1$ this implies
  \begin{equation}\label{eq:do:F}
  \begin{split}
    d^{(1)}(\cH^*) &= 1+\frac{2(r-1)(\ell-1)+1}{2(r-1)(2\ell-1)+2k(\ell-1)} \\
      & < 1+\frac{2(r-1)(\ell-1)+1}{2k(\ell-1)}
      \le 1+\frac{1}{2(r-1)(2\ell-1)} \\
      & \eqByRef{eq:do:D} d^{(1)}(\Drl) \,.
  \end{split}
  \end{equation}
  By~\eqref{eq:reserve:path} the hypergraph~$\cH^*$ satisfies~\ref{lem:reserve:2}.
  We define $\tilde I(Y,X)$ to be the reversal of $I(X,Y)$.
  It can be checked that~$\cH^*$ also contains the tight path
  \begin{multline*}
    \tilde U_A, \tilde I(A_1, U),
    \tilde U_B,\tilde I(B_1,A_1),
    \tilde A_{1,2},\tilde I(A_2,B_1),
    \tilde B_{1,2},\tilde I(B_2,A_2),
    \tilde A_{2,3}, \dots \\
    \dots, 
    \tilde A_{\ell-2,\ell-1}, \tilde I(A_{\ell-1},B_{\ell-2}),
    \tilde V_B, \tilde I(V,A_{\ell-1}),
    \tilde V_A \,.
  \end{multline*}
  This is a tight path from $\tpl{u}$ to $\tpl{v}$ running through all
  vertices of~$\cH^*$ but~$w^*$, and so~$\cH^*$ also satisfies~\ref{lem:reserve:3}. It
  remains to show that~$\Drl$ has maximal $1$-density among all subgraphs
  of~$\cH^*$.
  
  Suppose that $\cH$ is a subgraph of $\cH^*$ with maximal $1$-density. It
  follows that $\cH$ is an induced subgraph of $\cH^*$, and that we have
  $d^{(1)}(\cH)\ge d^{(1)}(\Drl)>1$. It follows that $\cH$ cannot contain
  any vertex of degree one (otherwise we could delete it and increase the
  $1$-density). In particular, this means that if $I(X,Y)$ is any of the
  sets of $k$ vertices which form a tight path in $\cH^*$ and which are not
  present in $\Drl$, then either every vertex of $I(X,Y)$ is in $\cH$, or
  none are. Similarly, by the definition of $k$ we have $k\cdot
  d^{(1)}(\Drl)>k+(r-1)$ and so $\cH$ cannot contain any $k$ vertices
  meeting only $k+r-1$ edges. Accordingly $\cH$ cannot contain $I(X,Y)$. It
  follows that $\cH$ must be a subgraph of $\Drl$.
  
  It is straightforward to check that if any of the vertices
  \begin{equation*}
  \begin{split}
    S:=
    \{u_2,\ldots,u_{2(r-1)-1},&w^*,v_2,\ldots,v_{2(r-1)-1},\\
    &a_2^{(i)},\ldots,
    a_{2(r-1)-1}^{(i)},b_2^{(i)},\ldots,b_{2(r-1)-1}^{(i)}\}
  \end{split}
  \end{equation*}
  of $\Drl$ is removed from $\Drl$, then we obtain a graph which can be
decomposed by successively removing vertices of degree at most one (i.e. it is
$1$-degenerate) and which therefore has $1$-density at most $1$. It follows that
$S\subset V(\cH)$. Now let $x$ be any
  vertex of $\Drl$ which is not in $\cH$. Since $x\not\in S$ we have
  $\deg_{\Drl}(x)=2$, and both edges containing $x$ have all their remaining
vertices in $S$. Thus we have
  \[d^{(1)}\big(\Drl\big[V(\cH)\cup\{x\}\big]\big)\ge\min\big(d^{(1)}(\cH),2\big)\]
  and since $d^{(1)}(\Drl)<2$, we conclude that $d^{(1)}(\cH)\le d^{(1)}(\Drl)$
  as required.
\end{proof}

\section{Concluding remarks}
\label{sec:conclude}

\paragraph{\bf Graphs}
We remark that our approach does not work (as such) in the case $r=2$, even for
the sub-optimal edge probability $n^{\eps-1}$. For this case, in the proof
of the connection lemma, Lemma~\ref{lem:connect}, when growing a fan we
would have to reveal in each iteration of the foreach-loop in
Algorithm~\ref{alg:fan} more than $n^{1-\eps}$ edges at a
vertex~$a$. In the construction of one fan 
we would have to repeat this operation at least $n^{(1/2)-2\eps}$ times:
only then we could hope for the fan to have $n^{(1/2)-\eps}$ leaves, which
we need in order to get a connection between two such fans at least in
expectation.  But then we would have revealed at least $n^{1-\eps}\cdot
n^{(1/2)-2\eps}=n^{(3/2)-3\eps}$ edges to obtain a single
connection. Hence, we cannot obtain a linear number of connections in this way, as
required by our strategy.

\smallskip

\paragraph{\bf Vertex disjoint cycles}
It is easy to modify our approach to show the following theorem.

\begin{theorem}\label{thm:factor} For every integer $r\ge 3$ and for every
  $\eps,\delta>0$ the following holds. Suppose that $n_1,\dots,n_\ell$ are
  integers, each at least $2r/\eps$, whose sum is at most $n$, and
  $n_1\ge\delta n$. Then for any $n^{-1+\eps}< p=p(n)\le 1$, the random
  $r$-uniform hypergraph $\cG^{(r)}(n,p)$ contains a collection of vertex
  disjoint tight cycles of lengths $n_1,\ldots,n_\ell$ with probability
  tending to one as $n$ tends to infinity.
\end{theorem}

A proof sketch is as follows. We refer to the steps used in the proof of
Theorem~\ref{thm:main}. 

First, we would run step~$1$ as before, except that we would find reservoir
graphs covering only at most $\eps\delta n/(8r)$ vertices. Step $2$ remains
unchanged. We would then in an extra step (requiring an extra round of
probability) to create greedily a collection of vertex disjoint tight paths
of lengths slightly shorter than $n_2,\ldots,n_\ell$, and another extra
step using Lemma~\ref{lem:connect} to connect these paths into tight cycles
of lengths $n_2,\ldots,n_\ell$. Here we require that the connecting paths
always have a precisely specified length.  As written,
Lemma~\ref{lem:connect} does not guarantee this (the output paths have
lengths differing by at most two, since the paths in each fan can differ in
length by one) but it is easy to modify the lemma to obtain this (we would
simply extend each of the shorter fan paths by one vertex while avoiding
dangerous sets). The remainder of the proof can remain almost unchanged. We
extend the reservoir path greedily to cover most of the remaining
vertices. Then we apply Lemma~\ref{lem:connect} twice to cover all the
leftover vertices and complete a cycle. Then this cycle has length $n_1$ as
desired. (The only difference is that some of our constants will need to be
adapted slightly.)

Again, for fixed $r$, $\eps$ and $\delta$ we obtain a randomised polynomial
time algorithm from this proof. Note that the condition that the cycles
should not be too short cannot be completely removed: in order to have
linearly many cycles of length $g$ with high probability, we require that
linearly many such cycles exist in expectation. This expectation is of the
order $n^gp^g$, which is in $o(n)$ if $p=o\big(n^{-(g-1)/g}\big)$.

\smallskip

\paragraph{\bf Derandomisation}
Our approach to Theorem~\ref{thm:main} yields a randomised
algorithm. However we only actually use the power of randomness in order to
preprocess our input hypergraph and `simulate' multi-round exposure. This
motivates the following question.

\begin{question}
  For a constructive proof which uses multi-round exposure, how can one
  obtain a \emph{deterministic} algorithm?
\end{question}

Replacing the randomised preprocessing step with a deterministic splitting
of the edges of the complete $r$-uniform hypergraph into disjoint dense
quasirandom subgraphs might be a promising strategy here.

Multi-Round exposure is a very common technique in probabilistic
combinatorics. Hence this question might be of interest for other problems
as well.

\smallskip

\paragraph{\bf Resilience.}
A very active recent development in the theory of random graphs is the
concept of resilience: under which conditions can one transfer a classical
extremal theorem to the random graph setting? Lee and Sudakov~\cite{LS},
improving on previous work of Sudakov and Vu~\cite{SV}, showed that Dirac's
theorem can be transferred to random graphs almost as sparse as at the
threshold for hamiltonicity. More precisely, they proved that for each
$\eps>0$, if $p\ge C\log n/n$ for some constant $C=C(\eps)$, then almost
surely the random graph $G=G(n,p)$ has the following property. Every
spanning subgraph of $G$ which has minimum degree $(\tfrac{1}{2}+\eps)pn$,
contains a Hamilton cycle.

It would be interesting to prove a corresponding result for tight Hamilton cycles
in subgraphs of random hypergraphs. It is unlikely that the Second Moment
Method will provide help for this. Our methods, however, might be robust
enough to provide some assistance.

\section{Acknowledgements}

We would like to thank Klas Markstr\"om for suggesting
 Theorem~\ref{thm:factor}.


\bibliographystyle{amsplain_yk} 
\bibliography{TightCycle}

\end{document}